\documentclass{amsart}
\usepackage{mathrsfs}
\usepackage{amsfonts}
\usepackage{amssymb}
\usepackage{amsxtra}
\usepackage[english,polish,german]{babel}

\newcommand{\pl}[1]{\foreignlanguage{polish}{#1}}

\usepackage[colorlinks,citecolor=blue,urlcolor=blue,bookmarks=true]{hyperref}
\hypersetup{
pdfpagemode=UseNone,
pdfstartview=FitW,
pdfdisplaydoctitle=true,
pdfborder={0 0 0}, 
pdftitle={Long time behaviour of random walks on the integer lattice},
pdfauthor={Bartosz Trojan},
pdflang=en-GB
}

\newcommand{\ZZ}{\mathbb{Z}}
\newcommand{\RR}{\mathbb{R}}
\newcommand{\CC}{\mathbb{C}}
\newcommand{\NN}{\mathbb{N}}

\newcommand{\calV}{\mathcal{V}}
\newcommand{\calM}{\mathcal{M}}

\newcommand{\calU}{\mathcal{U}}
\newcommand{\calF}{\mathcal{F}}
\newcommand{\calO}{\mathcal{O}}

\newcommand{\sprod}[2]{{\langle #1, #2\rangle}}
\newcommand{\norm}[1]{{\left\lvert #1 \right\rvert}}
\newcommand{\vnorm}[1]{{\left\lVert {#1} \right\rVert}}
\newcommand{\abs}[1]{{\lvert {#1} \rvert}}
\newcommand{\der}[1]{\partial^{#1}}
\newcommand{\dth}{{\: \rm d}\theta}
\newcommand{\sign}{\operatorname{sign}}

\newcommand{\Log}{\operatorname{Log}}

\newcommand{\conv}{\operatorname{conv}}

\newcommand{\dist}{\operatorname{dist}}

\renewcommand{\atop}[2]{\substack{{#1}\\{#2}}}

\theoremstyle{plain}
\newtheorem{theorem}{Theorem}[section]

\newtheorem{lemma}[theorem]{Lemma}
\newtheorem{claim}{Claim}
\newtheorem{corollary}[theorem]{Corollary}

\theoremstyle{plain}
\newcounter{thm}

\newtheorem{main_theorem}[thm]{Theorem}

\theoremstyle{remark}
\newtheorem{remark}{Remark}
\newtheorem{example}{Example}

\begin{document}
\selectlanguage{english}

\title[Long time behaviour of random walks]
{Long time behaviour of random walks on the integer lattice}

\author{Bartosz Trojan}

\address{
    Instytut Matematyczny\\
	Uniwersytet \pl{Wroc"lawski}\\
	Pl. Grun\-waldzki 2/4\\
	50-384 \pl{Wroc"law}\\
	Poland}
\email{trojan@math.uni.wroc.pl}

\begin{abstract}
	We consider an irreducible finite range random walk on the $d$-dimensional integer lattice and study asymptotic
	behaviour of its transition function $p(n; x)$. In particular, for simple random walk our asymptotic formula
	is valid as long as $n (n - \norm{x}_1)^{-2}$ tends to zero.
\end{abstract}



\maketitle

\section{Introduction}
In 1921, with the article \cite{polya0} P\'olya pioneered research on the simple random walk on the integer lattice.
Using Fourier analysis he proved that $p(n; x)$, the $n$'th step transition function, satisfies
\footnote{For $x \in \RR^d$ and $p \in (1, \infty)$ we set
$\norm{x}_p = \big(\abs{x_1}^p + \cdots + \abs{x_d}^p\big)^{1/p}$} 
\begin{align*}
	\lim_{n \to +\infty} (2 n)^{\frac{d}{2}} p(2 n; x) &= 2 d^{\frac{d}{2}} (2 \pi)^{-\frac{d}{2}}, \quad
	\text{if } \norm{x}_1 \equiv 0 \pmod 2,\\
	\lim_{n \to +\infty} (2n-1)^{\frac{d}{2}} p(2n-1; x) &= 2 d^{\frac{d}{2}} (2 \pi)^{-\frac{d}{2}}, \quad
	\text{if } \norm{x}_1 \equiv 1 \pmod 2,
\end{align*}
for any $x \in \ZZ^d$. Essentially, P\'olya's proof shows
that (see Spitzer \cite[Remark after P7.9]{sp})
\[
	p(n; x) = 
	\begin{cases}
		2 d^{\frac{d}{2}} (2  \pi n)^{-\frac{d}{2}} e^{-\frac{d}{2 n} \norm{x}_2^2} + o\big(n^{-\frac{d}{2}} \big)
		& \text{if } \norm{x}_1 \equiv n \pmod 2,\\
		0 & \text{otherwise,}
	\end{cases}
\]
uniformly with respect to $n \in \NN$ and $x \in \ZZ^d$. As it may be easily seen the local limit theorem is very
inaccurate if $\norm{x}_2$ is larger than $\sqrt{n}$. Further development of the Fourier method allowed to gain better
control over the error term for large $\norm{x}_2$ (see Smith \cite{smith}, Spitzer \cite[P7.10]{sp}, Ney and Spitzer
\cite[Theorem 2.1]{ns}). Namely,
\begin{align*}
	p(n; x) = 
	\begin{cases}
		2 d^{\frac{d}{2}} (2 \pi n)^{-\frac{d}{2}} e^{-\frac{d}{2 n} \norm{x}_2^2} 
		+ o\big(n^{-\frac{d}{2}+1} \norm{x}_2^{-2} \big),
		& \text{if } \norm{x}_1 \equiv n \pmod 2,\\
		0 & \text{otherwise,}
	\end{cases}
\end{align*}
uniformly with respect to $x \in \ZZ^d \setminus \{0\}$. Let us observe that the error in the approximation of $p(n; x)$
is additive and may become big compared to the first term. In many applications, it is desired to have an asymptotic
formula for $p(n; x)$ valid on the largest possible region with respect to $n$ and $x$. There are some results in this
direction available. In particular, (see Lawler \cite[Propositon 1.2.5]{law}, Lawler and Limi\v{c} 
\cite[Theorem 2.3.11]{lawlim}) there is $\rho > 0$ such that for all $n \in \NN$ and $x \in \ZZ^d$,
if $\norm{x}_2 \leq \rho n$ then
\begin{align*}
	p(n; x) =
	\begin{cases}
		d^{\frac{d}{2}} (2 \pi n)^{-\frac{d}{2}} e^{-\frac{d}{2n} \norm{x}_2^2}
		\big(2 + \calO(n^{-1}) + \calO(n^{-3} \norm{x}_2^4)\big), 
		& \text{if } \norm{x}_1 \equiv n \pmod 2,\\
		0 & \text{otherwise.}
	\end{cases}
\end{align*}
We want to emphasize that the above asymptotic formula is useful only in the region where $\norm{x}_2 = o(n^{3/4})$.
Therefore, there arises a natural question:
\begin{quote}
	\em
	Is there an asymptotic formula for $p(n; x)$ which is valid on a larger region than $\norm{x}_2 = o(n^{3/4})$? 
\end{quote}
The subject of the present article is to give a positive answer to the posed question. Although, the asymptotic
will be formulated for the simple random walk, the actual result is valid for any irreducible finite range random
walk with the mean zero or not (see Theorem \ref{th:7}). Before we state the main theorem, let us introduce some notation.
For $\delta \in \calM = \big\{x \in \RR^d : \norm{x}_1 < 1 \big\}$ we set
\begin{equation}
	\label{eq:41}
	\phi(\delta) = \max\big\{\sprod{x}{\delta} - \log \kappa(x) : x \in \RR^d \big\}
\end{equation}
where $\kappa$ is a function on $\RR^d$ defined by
\[
	\kappa(x) = \frac{1}{d} \big(\cosh x_1 + \cdots + \cosh x_d\big).
\]
In Section \ref{sec:2} we prove the following theorem.
\begin{main_theorem}
\label{th:3}
For all $x \in \ZZ^d$ and $n \in \NN$, if $\norm{x}_1 \equiv n \pmod 2$ then
\begin{equation}
	\label{eq:8} 
	p(n; x) = 
	(2 \pi n)^{-\frac{d}{2}} (\det B_s\big)^{-\frac{1}{2}} e^{-n \phi(\delta)}
	\big(2 + \calO\big(n^{-1} (1 - \norm{\delta}_1)^{-2}\big)\big)
\end{equation}
otherwise, $p(n; x) = 0$, where $\delta = \frac{x}{n}$ and $s=\nabla \phi(\delta)$.
\end{main_theorem}
Some comments are in order. First, observe that the asymptotic formula \eqref{eq:8} is valid in a region excluding
only the case when $n (1 - \norm{\delta}_1)^2$ stays bounded. Although, the function $\phi$ is positive
convex and comparable to $\norm{\: \cdot \:}_2^2$, it cannot be replaced in the asymptotic formula
by $\norm{\: \cdot \:}_2^2$ without introducing an additional error term, see Remark \ref{rem:1}. For processes with
continuous time it was observed by Davis in \cite{dav} that in order to get an upper bound for the heat kernel on
a larger region one has to introduce a non-Gaussian factor. Therefore, Theorem \ref{th:3} may be considered as a
discrete time counterpart of \cite{dav}. Finally, however the quadratic form $B_x$ is given explicitly by \eqref{eq:39},
the mapping $\calM \ni \delta \mapsto s(\delta)$ is an implicit function. We want to stress the fact that when
$\norm{\delta}_1$ approaches one, the value of $\norm{s}_2$ tends to infinity. In particular, the quadratic form $B_s$
degenerates. For this reason a more convenient form of Theorem \ref{th:3} is given in Corollary \ref{cor:1}.
Namely, for all $\epsilon > 0$, $x \in \ZZ^d$ and $n \in \NN$, if $\norm{x}_1 \leq n(1-\epsilon)$ then
\begin{equation*}
	p(n; x) = 
	\begin{cases}
	d^{\frac{d}{2}} (2 \pi n)^{-\frac{d}{2}} e^{-n\phi(\delta)} 
	\big(2 + \calO(\norm{\delta}_1) + \calO(n^{-1})\big), & \text{if } \norm{x}_1 \equiv n \pmod 2,\\
	0 & \text{otherwise.}
	\end{cases}
\end{equation*}
The last asymptotic formula is useful as long as $\norm{x}_1 = o(n)$.

Let us comment about the method of the proof. First, with a help of the Fourier inversion formula, we write $p(n; x)$
as an oscillatory integral. We split the integral into two parts. The first part we analyse by the Laplace method. This is
not a straightforward application of it, since the phase function degenerates as $\norm{\delta}_1$ approaches one. 
To estimate the second part we develop a geometric argument, which allows us to control the way how the quadratic form
$B_s$ degenerates.

The result obtained in this article has already found an application in the study of subordinated random walks
(see \cite{bct}) which are spread over all $\ZZ^d$ and do not have second moment. Also the geometric method developed here
can be successfully applied in much wider context. Namely, to study finitely supported isotropic random walks on affine
buildings (see \cite{tr}). There is also an ongoing project to get
the precise asymptotic formula for random walks with internal degrees of freedom extending the one obtained by 
Kr\'amli and Sz\'asz \cite{kSz} (see also Guivarc'h \cite{guiv}). Finally, Appendix \ref{sec:1} contains application of
Theorem \ref{th:7} to triangular and hexagonal lattices. This has to be compared with results recently obtained in
\cite{ks1, ks2, ikk}.

\subsection{Notation}
We use the convention that $C$ stands for a generic positive constant whose value can change
from line to line. The set of positive integers is denoted by $\NN$. Let $\NN_0 = \NN \cup \{0\}$.

\section{Preliminaries}

\subsection{Random walks}
Let $p(\cdot, \cdot)$ denote the transition density of a random walk on the $d$-dimensional integer lattice.
Let $p(x) = p(0, x)$. For $n \in \NN_0$ and $x \in \ZZ^d$ we set
\[
	p(n+1; x) = \sum_{y \in \ZZ^d} p(n; y)  p(x - y)
\]
and $p(1; x) = p(x)$. The support of $p$ is denoted by $\calV$, i.e.
\[
	\calV = \big\{v \in \ZZ^d : p(v) > 0 \big\}.
\]
Let $\kappa: \CC^d \rightarrow \CC$ be an exponential polynomial defined by
\[
	\kappa(z) = \sum_{v \in \calV} p(v) e^{\sprod{z}{v}},
\]
where $\sprod{\:\cdot\:}{\:\cdot\:}$ is the standard scalar product on $\CC^d$
\[
	\sprod{z}{w} = \sum_{j = 1}^d z_j \overline{w_j}.
\]
In particular, $\RR^d \ni \theta \mapsto \kappa(i\theta)$ is the characteristic function of $p$. We set
\begin{equation}
	\label{eq:43}
	\calU = \big\{\theta \in [-\pi, \pi)^d : \abs{\kappa(\theta)} = 1 \big\}.
\end{equation}
Finally, the interior of the convex hull of $\calV$ in $\RR^d$ is denoted by $\calM$. 

In this article, we study the asymptotic behaviour of transition functions of irreducible random walks. Let us recall
that the random walk is \emph{irreducible} if for each $x \in \ZZ^d$ there is $n \in \NN$ such that $p(n; x) > 0$.
By $d \in \NN$ we denote the period of $p$, that is 
\[
	d = \gcd\big\{n \in \NN : p(n; 0) > 0\big\}.
\]
Then the space $\ZZ^d$ decomposes into $r$ disjoint classes
\[
	X_j = \big\{x \in \ZZ^d : p(j + k r; x) > 0 \text{ for some } k \geq 0 \big\}
\]
for $j = 0, \ldots, r-1$. We observe that for $j \in \{0, \ldots, r-1\}$ and $x \in X_j$, if $n \not \equiv j \pmod r$
then
\[
	p(n; x) = 0.
\]
For each $x \in \ZZ^d$, by $m_x$ we denote the smallest $m \in \NN$ such that $p(m; x) > 0$, thus
$x / m_x \in \overline{\calM}$. Notice that there is $C \geq 1$ such that for all $x \in \ZZ^d$
\begin{equation}
	\label{eq:42}
	C^{-1} \norm{x}_1 \leq m_x \leq C \norm{x}_1.
\end{equation}
Indeed, let $\{e_1, \ldots, e_d\}$ be the standard basis of $\RR^d$. Since
\[
	e_j = \sum_{v \in \calV} m_{j, v} v,
	\quad \text{ and } \quad
	-e_j = \sum_{v \in \calV} m_{-j, v} v,
\]
for some $m_{j, v}, m_{-j, v} \in \NN_0$ satisfying
\[
	m_{e_j} = \sum_{v \in \calV} m_{j, v},
	\quad \text{ and } \quad
	m_{-e_j} = \sum_{v \in \calV} m_{-j, v},
\]
by setting $\varepsilon_j = \sign \sprod{x}{e_j}$ we get
\[
	x = \sum_{v \in \calV} \Big(\sum_{j = 1}^d m_{\varepsilon_j j, v} \abs{\sprod{x}{e_j}} \Big) v.
\]
Hence,
\[
	m_x \leq \norm{x}_1 \sum_{j = 1}^d \big(m_{e_j} + m_{-e_j}\big).
\]
which, together with boundedness of $\overline{\calM}$, implies \eqref{eq:42}.

Next, we observe that there is $K > 0$ such that for all $k \geq K$ 
\[
	p(k r; 0) > 0,
\]
thus for all $x \in \ZZ^d$ and $n \geq Kr + m_x$
\begin{equation}
	\label{eq:6}
	p(n; x) =
	\begin{cases}
	> 0 & \text{ if } n \equiv m_x \pmod r, \\
	= 0 & \text{ otherwise.}
	\end{cases}
\end{equation}
Since
\begin{equation}
	\label{eq:3}
	\bigg\{\frac{e_1}{m_{e_1}}, -\frac{e_1}{m_{-e_j}}, \ldots, \frac{e_d}{m_{e_d}}, -\frac{e_d}{m_{-e_d}} \bigg\}
\end{equation}
do not lay on the same affine hyperplane, the interior of the convex hull of \eqref{eq:3} is a non-empty 
subset of $\calM$.

For each $x \in \RR^d$, by $B_x$ we denote a quadratic form on $\RR^d$ defined by
\begin{equation}
	\label{eq:39}
	B_x(u, u)
	=
	D_u^2 \log \kappa(x),
\end{equation}
where $D_u$ denotes the derivative along a vector $u$, i.e.
\[
	D_u f(x) = \left. \frac{{\rm d}}{{\rm d}t} f(x + t u) \right|_{t = 0}.
\]
Since
\[
	D_u \log \kappa(x) = \sum_{v \in \calV} \frac{p(v) e^{\sprod{x}{v}}}{\kappa(x)} \sprod{u}{v}
\]
and
\[
	D_u \left( \frac{p(v) e^{\sprod{x}{v}}}{\kappa(x)} \right)
	=
	\frac{p(v) e^{\sprod{x}{v}}}{\kappa(x)} \sprod{u}{v}
	-
	\sum_{v' \in \calV} 
	\frac{p(v) e^{\sprod{x}{v}}}{\kappa(x)} 
	\cdot
	\frac{p(v') e^{\sprod{x}{v'}}}{\kappa(x)} 
	\sprod{u}{v'},
\]
we may write
\begin{equation}
	\label{eq:1}
	B_x(u, u) = \frac{1}{2} \sum_{v, v'} 
	\frac{p(v) e^{\sprod{x}{v}}}{\kappa(x)}
	\cdot 
	\frac{p(v') e^{\sprod{x}{v'}}}{\kappa(x)}
	\sprod{u}{v - v'}^2.
\end{equation}
In particular, if the random walk is irreducible then the quadratic form $B_x$ is positive definite.
\begin{example}
	\label{ex:1}
	Let $p$ be the transition function of the simple random walk on $\ZZ^d$, i.e.
	\[
		p(e_j) = p(-e_j) = \frac{1}{2d}, \quad \text{for} \quad j = 1, \ldots, d.
	\]
	Thus
	\[
		\calV = \{ \pm e_j : j =1, \ldots, d\},
	\quad\text{and}\quad
		\calM = \big\{x \in \RR^d : \norm{x}_1 < 1\big\}.
	\]
	Since
	\[
		\kappa(z) = \frac{1}{2d} \sum_{j=1}^d \big(e^{z_j} + e^{-z_j}\big),
	\]
	we get $\calU = \{0, (-\pi, -\pi, \ldots, -\pi)\}$.
	By a straightforward calculation we may find the quadratic form $B_0$,
	\[	
		B_0(u, u)
		=
		\frac{1}{(2 d)^2}
		\sum_{j=1}^d \sum_{j' = 1}^d (u_j + u_{j'})^2 + (u_j - u_{j'})^2
		=
		\frac{1}{d} \sprod{u}{u}.
	\]
\end{example}

\subsection{Function $s$}
For the sake of completeness we provide the proof of the following well-known theorem.
\begin{theorem}
	\label{th:1}
	For every $\delta \in \calM$ a function $f(\delta,\:\cdot\:): \RR^d \rightarrow \RR$ defined by
	$$
	f(\delta, x) = \sprod{x}{\delta} - \log \kappa(x)
	$$
	attains its maximum at the unique point $s \in \RR^d$ satisfying $\nabla \log \kappa(s) = \delta$.
\end{theorem}
\begin{proof}
	Without loss of generality, we may assume $\nabla \kappa(0)=0$. Indeed, otherwise we will consider
	\[
		\tilde{\kappa}(z)
		= e^{-\sprod{z}{v_0}} \kappa(z) 
		= \sum_{v \in \tilde{\calV}} p(v + v_0) e^{\sprod{z}{v}}
	\]
	where $v_0 = \nabla \kappa(0)$ and $\tilde{\calV} = \calV - v_0$. Then
	$\tilde{\calM}$, the interior of the convex hall of $\tilde{\calV}$, is equal to $\calM - v_0$.  For
	$\tilde{\delta} = \delta - v_0$ we have
	\[
		\tilde{f}(\tilde{\delta}, x) = \sprod{x}{\delta - v_0} - \log \tilde{\kappa}(x) 
		=
		\sprod{x}{\delta} - \log \kappa(x)
		=
		f(\delta, x).
	\]
	We conclude that if $s$ is the unique maximum of $\RR^d \ni x \mapsto \tilde{f}(\tilde{\delta}, x)$, then 
	it is also the unique maximum of $\RR^d \ni x \mapsto f(\delta, x)$. Because
	\[
		\nabla \log \tilde{\kappa}(x) = \nabla \log \kappa(x) - v_0
	\]
	we get $\nabla \log \kappa(s) = \tilde{\delta} + v_0 = \delta$, proving the claim.
	
	Fix $\delta \in \calM$. Since $\nabla \kappa(0) = 0$, by Taylor's theorem we have
	\[
		f(\delta, x) =  \sprod{x}{\delta} + \mathcal{O}(\norm{x}_2^2)
	\]
	as $\norm{x}_2$ approaches zero. Moreover, for any $x, u \in \RR^d$ 
	\[
		D_u^2 f(\delta, x) = -B_x(u, u),
	\]
	thus the function $\RR^d \ni x \mapsto f(\delta, x)$ is strictly concave.

	Let us observe that 
	\[
		0 = \nabla \kappa(0) = \sum_{v \in \calV} p(v) \cdot v \in \overline{\calM}.
	\]
	Since $\calM$ is not empty, the set $\calV$ cannot be contained in an affine hyperplane, thus, $0 \in \calM$.
	
	Now, $\delta \in \calM$ implies that there are $v_1, \ldots, v_d \in \partial \calM \cap \calV$ such that $\delta$
	belongs to the convex hull of $\{0, v_1, \ldots, v_d\}$, i.e. there are $t_0, t_1, \ldots t_d \in [0, 1]$
	satisfying
	\[
		\delta = t_0 \cdot 0 + \sum_{j=1}^d t_j \cdot v_j = \sum_{j=1}^d t_j \cdot v_j,
	\]
	Because $\delta \not \in \partial \calM$ we must have $t_0 > 0$, thus $\sum_{j = 1}^d t_j < 1$. Hence,
	\[
		\sum_{j=1}^d t_j \log \kappa(x) 
		\geq \sum_{j=1}^d t_j \big(\log p(v_j) + \sprod{x}{v_j}\big)
		=\sum_{j=1}^d t_j \log p(v_j) + \sprod{x}{\delta},
	\]
	and we get
	\[
		f(\delta, x) = 
		\sprod{x}{\delta} - \log \kappa(x)
		\leq \Big(\sum_{j=1}^d t_j - 1\Big) \log \kappa(x) - \sum_{j=1}^d t_j \log p(v_j),
	\]
	which implies that
	\[
		\lim_{\norm{x}_2 \to \infty} f(\delta, x) = -\infty,
	\]
	because
	\[
		\lim_{\norm{x}_2 \to \infty} \log \kappa(x) = +\infty,
	\]
	and the proof is finished.
\end{proof}
In the rest of the article, given $\delta \in \calM$ by $s$ we denote the unique solution to
\begin{equation}
	\label{eq:16}
	\delta = \nabla \log \kappa(s) = \sum_{v \in \calV} \frac{p(v) e^{\sprod{s}{v}}}{\kappa(s)} \cdot v.
\end{equation}
Let $\phi: \mathcal{M} \rightarrow \mathbb{R}$ be defined by
\begin{equation}
	\label{eq:30}
	\phi(\delta) = \max\{\sprod{x}{\delta} - \log \kappa(x) : x \in \RR^d\},
\end{equation}
thus, by Theorem \ref{th:1},
\begin{equation}
	\label{eq:19}
	\phi(\delta) = \sprod{\delta}{s} - \log \kappa(s).
\end{equation}
By \eqref{eq:16}, for any $u \in \RR^d$,
\[ 
	\sprod{\delta}{u} = D_u \log \kappa(s).
\]
Hence, for $u, u' \in \RR^d$
\[
	\sprod{u}{u'} 
	=D_u \big(D_{u'} \log \kappa(s) \big)
	=\sum_{j = 1}^d D_j D_{u'} \log \kappa(s) D_u s_j
	= B_s(D_u s, u'),
\]
i.e, $D_u s = B_s^{-1} u$. Therefore, we can calculate
\[ 
	\nabla \phi(\delta) 
	= s + \sum_{j = 1}^d \delta_j \nabla s_j - \sum_{j =1}^d D_j \log \kappa(s) \nabla s_j
	= s,
\]
thus, 
\[
	D^2_u \phi(\delta) = D_u \big( \sprod{u}{s} \big) = B_s^{-1}(u,u).
\]
In particular, $\phi$ is a convex function on $\calM$. Let $\delta_0 = \nabla \log \kappa(0)$. By Taylor's theorem,
we have
\begin{equation}
	\label{eq:10}
	\phi(\delta) = \frac{1}{2} B_0^{-1}(\delta - \delta_0, \delta-\delta_0) +
	\mathcal{O}(\norm{\delta-\delta_0}_2^3)
\end{equation}
as $\delta$ approaches $\delta_0$. We claim that
\begin{claim}
	For all $\delta \in \calM$
	\footnote{$A \asymp B$ means that $c B \leq A \leq C B$, for some constants $c,C>0$.}
	\[
		\phi(\delta) \asymp B_0^{-1}(\delta - \delta_0, \delta - \delta_0).
	\]
\end{claim}
Since $\phi$ is convex and satisfies \eqref{eq:10}, it is enough to show that $\phi$ is bounded from above.
Given $\delta \in \calM$, let $v_0 \in \calV$ be any vector satisfying
\[
	\sprod{s}{v_0} = \max\big\{\sprod{s}{v} : v \in \calV\big\}.
\]
Because
$$
\sprod{s}{\delta} - \sprod{s}{v_0} =
\sum_{v \in \mathcal{V}}
\frac{p(v) e^{\sprod{s}{v}} }{\kappa(s)}\sprod{s}{v-v_0} \leq 0,
$$
we get 
\begin{align*}
	\phi(\delta) = \sprod{s}{\delta} - \log \kappa(s) 
	& \leq \sprod{s}{\delta} - \log \big(p(v_0) e^{\sprod{s}{v_0}}\big) \\
	& \leq -\log p(v_0),
\end{align*}
proving the claim.

\begin{example}
	Let $p$ be the transition density of the simple random walk on $\ZZ$. Then $\calV = \{-1, 1\}$,
	$\calU = \{0, -\pi\}$ and $\calM = (-1, 1)$. For $\delta \in \calM$, we have
	\[
		e^s = \sqrt{\frac{1 + \delta}{1 - \delta}},
	\]
	and
	\[
		\kappa(s) = \frac{e^s + e^{-s}}{2} = \frac{1}{\sqrt{1-\delta^2}}.
	\]
	Hence, using \eqref{eq:19} we obtain
	\[
		\phi(\delta) = \frac{1}{2} (1-\delta) \log (1-\delta) + \frac{1}{2}(1+\delta) \log (1+\delta).
	\]
\end{example}
In general, there is no explicit formula for the function $\phi$. By implicit function theorem, the function
$s$ is real analytic on $\calM$. In particular, $s$ is bounded on any compact subset of $\calM$. From the other
side, $\norm{s}_2$ approaches infinity when $\delta$ tends to $\partial \calM$.
To see this, denote by $\calF$ a facet of $\calM$ such that $\delta$ approaches $\partial \calM \cap \calF$. Let
$u$ be an outward unit normal vector to $\calM$ at $\calF$. Then for each $v_1 \in \calF \cap \calV$ and 
$v_2 \in \calV \setminus \calF$ we have
\begin{align*}
	\sprod{v_1 - \delta}{u} 
	& = \sum_{v \in \calV} \frac{p(v) e^{\sprod{s}{v}}}{\kappa(s)} \sprod{v_1 - v}{u} \\
	& = \sum_{v \in \calV \setminus \calF} \frac{p(v) e^{\sprod{s}{v}}}{\kappa(s)} \sprod{v_1 - v}{u}
	\geq \frac{p(v_2) e^{\sprod{s}{v_2}}}{\kappa(s)} \sprod{v_1 - v_2}{u}.
\end{align*}
Therefore, for any $v \in \calV \setminus \calF$
\begin{equation}
	\label{eq:7}
	\lim_{\delta \to \partial \calM \cap \calF} \frac{e^{\sprod{s}{v}}}{\kappa(s)} = 0.
\end{equation}
The next theorem provides a control over the speed of convergence in \eqref{eq:7}. 
\begin{theorem}
	\label{th:2}
	There are constants $\eta \geq 1$ and $C > 0$ such that for all $\delta \in \mathcal{M}$, and $v \in \mathcal{V}$
	we have
	\[
		\frac{e^{\sprod{s}{v}}}{\kappa(s)}
		\geq C \dist(\delta, \partial\mathcal{M})^\eta
	\]
	where $s=s(\delta)$ satisfies $\delta = \nabla \log \kappa(s)$. 
\end{theorem}
\begin{proof}
	We consider any enumeration of elements of $\calV = \{v_1, \ldots, v_N\}$. Define
	$$
	\Omega = \big\{\omega \in \mathcal{S}:
		\sprod{\omega}{v_{i}} \geq \sprod{\omega}{v_{i+1}} \text{ for } i=1,\ldots,N-1 \big\}
	$$
	where $\mathcal{S}$ is the unit sphere in $\RR^d$ centred at the origin. Suppose
	$\Omega \neq \emptyset$ and let $k$ be the smallest index such that points $\{v_1, \ldots, v_k\}$ do not lay on the
	same facet of $\calM$. Let us recall that a set $\calF$ is a facet of $\calM$ if there is $\lambda \in \mathcal{S}$
	and $c \in \RR$ such that for all $v \in \calV$, $\sprod{\lambda}{v} \leq c$, and
	\[
		\calF = \conv\{v \in \calV : \sprod{\lambda}{v} = c\}.
	\]
	Since $\{v_1, \ldots, v_k\}$ do not lay on the same facet of $\calM$ and $\Omega$ is a compact set, there is 
	$\epsilon > 0$ such that for all $\omega \in \Omega$ we have
	\begin{equation}
		\label{eq:24}
		\sprod{\omega}{v_1} \geq \sprod{\omega}{v_k} + \epsilon.
	\end{equation}
	Let $\calF$ be a facet containing $\{v_1,\ldots,v_{k-1}\}$. For $\frac{x}{\norm{x}_2} \in \Omega$ and
	$$
	\delta = \sum_{v\in\mathcal{V}} \frac{p(v) e^{\sprod{x}{v}}}{\kappa(x)} \cdot v,
	$$
	we have
	\begin{align*}
		\dist(\delta, \partial \mathcal{M}) 
		\leq 
		\sprod{\lambda}{v_1 - \delta}
		& = 
		\sum_{v\in\mathcal{V}\setminus\mathcal{F}} 
		\frac{p(v) e^{\sprod{x}{v}}}{\kappa(x)}
		\sprod{\lambda}{v_1-v} \\
		& \leq 
		C \frac{e^{\sprod{x}{v_k}}}{\kappa(x)}.
	\end{align*}
	Since
	\[
		p(v_1) e^{\sprod{x}{v_1}} \leq \kappa(x) \leq e^{\sprod{x}{v_1}}
	\]
	we obtain 
	\[
		\dist(\delta, \partial \calM) \leq C e^{\sprod{x}{v_k - v_1}}.
	\]
	In particular, for $1 \leq j \leq k$, we have
	\[
		\frac{e^{\sprod{x}{v_j}}}{\kappa(x)} \geq C \dist(\delta, \partial \calM).
	\]
	If $j > k$, we can estimate
	\begin{align*}
		\frac{e^{\sprod{x}{v_j}}}{\kappa(x)}
		\geq
		e^{\sprod{x}{v_j - v_1}} 
		& = 
		\Big(e^{\sprod{x}{v_k - v_1}}\Big)^{\sprod{x}{v_1 - v_j}/\sprod{x}{v_1 - v_k}} \\
		& \geq C \dist(\delta, \partial\mathcal{M})^{\sprod{x}{v_1-v_j}/\sprod{x}{v_1 - v_k}}
	\end{align*}
	what finishes the proof since by \eqref{eq:24}
	\[
		1 \leq \frac{\sprod{x}{v_1 - v_j}}{\sprod{x}{v_1 - v_k}} \leq \epsilon^{-1} \norm{v_1 - v_j}_2.
		\qedhere
	\]
\end{proof}

\subsection{Analytic lemmas}
For a multi-index $\sigma \in \mathbb{N}^d$ we denote by $X_\sigma$ a multi-set containing
$\sigma(i)$ copies of $i$. Let $\Pi_\sigma$ be a set of all partitions of $X_\sigma$.
For the convenience of the reader we recall the following lemma.

\begin{lemma}[Fa\`a di Bruno's formula]
	\label{lem:1}
	There are positive constants $c_\pi$, $\pi \in \Pi_\sigma$, such that for sufficiently smooth
	functions $f: S \rightarrow T$, $F: T \rightarrow \mathbb{R}$, $T \subset \mathbb{R}$,
	$S \subset \mathbb{R}^d$, we have
	$$
	\der{\sigma} F(f(s)) = \sum_{\pi \in \Pi_\sigma} c_\pi \left.
	\frac{{\rm d}^m }{{\rm d} t^m} \right|_{t = f(s)} F(t)
	\prod_{j=1}^{m} \der{B_j} f(s)
	$$
	where $\pi = \{B_1, \ldots, B_m\}$.
\end{lemma}

For a multi-set $B$ containing $\sigma(i)$ copies of $i$ we set $B! = \sigma!$. Let us
observe that for
$$
F(t) = \frac{1}{2-t}, \quad \text{and}\quad f(s) = \prod_{j=1}^d \frac{1}{1 - s_j},
$$
the function $F(f(s))$ is real-analytic in some neighbourhood of $s = 0$. Hence, there is $C > 0$
such that for every $\sigma \in \mathbb{N}^d$
\begin{equation}
	\label{eq:22}
	\der{\sigma} F(f(0)) = 
	\sum_{\pi \in \Pi_\sigma} c_\pi m! \prod_{j=1}^m B_j! \leq C^{\abs{\sigma}+1} \sigma!.
\end{equation}
Using Lemma \ref{lem:1} one can show
\begin{lemma}
	\label{lem:2}
	Let $\calV \subset \RR^d$ be a set of finite cardinality. Assume that for each $v \in \calV$, 
	we are given $a_v \in \mathbb{C}$, and $b_v > 0$. Then for $z = x+i\theta \in \CC^d$ such that
	\[
		\norm{\theta}_2 \leq (2 \cdot \max\{\norm{v}_2: v \in \mathcal{V}\})^{-1},
	\]
	we have
	\begin{equation}
		\label{eq:11}
		\Big| \sum_{v \in \mathcal{V}} b_v e^\sprod{z}{v} \Big|
		\geq \frac{1}{\sqrt{2}} \sum_{v \in \mathcal{V}} b_v e^\sprod{x}{v}.
	\end{equation}
	Moreover, there is $C > 0$ such that for all $\sigma \in \mathbb{N}^d$
	\begin{equation}
		\label{eq:12}
		\bigg| \der{\sigma} \bigg\{ \frac{ \sum_{v \in \mathcal{V}} a_v e^\sprod{z}{v} }
		{ \sum_{v \in \mathcal{V}} b_v e^\sprod{z}{v} } \bigg\} \bigg| 
		\leq C^{\abs{\sigma}} \sigma! \frac{ \sum_{v \in \mathcal{V}} \abs{a_v} e^\sprod{x}{v} }
		{ \sum_{v \in \mathcal{V}} b_v e^\sprod{x}{v} }.
	\end{equation}
\end{lemma}
\begin{proof}
	We start by proving \eqref{eq:11}. We have
	\begin{align*}
		\Big| \sum_{v \in \calV} b_v e^{\sprod{z}{v}} \Big|^2
		& =
		\sum_{v, v' \in \calV} b_v b_{v'} e^{\sprod{x}{v + v'}} \cos\sprod{\theta}{v-v'} \\
		& \geq
		\sum_{v, v' \in \calV} b_v b_{v'} e^{\sprod{x}{v + v'}} \left(1 - \frac{\sprod{\theta}{v-v'}^2}{2}\right) \\
		& \geq
		\frac{1}{2}
		\Big( \sum_{v \in \calV} b_v e^{\sprod{x}{v}} \Big)^2 
	\end{align*}
	because $\abs{\sprod{\theta}{v-v'}} \leq 1$.

	For the proof of \eqref{eq:12}, it is enough to show
	\begin{equation}
		\label{eq:26}
		\bigg|\partial^{\sigma} \bigg\{ \frac{1}{\sum_{v \in \calV} b_v e^{\sprod{z}{v}}}\bigg\} \bigg|
		\leq
		C^{\abs{\sigma} + 1} \sigma! \frac{1}{\sum_{v \in \calV} b_v e^{\sprod{x}{v}}}.
	\end{equation}
	Indeed, since
	\begin{equation}
		\label{eq:28}
		\Big|
		\partial^\alpha \Big\{\sum_{v \in \calV} a_v e^{\sprod{z}{v}} \Big\}
		\Big|
		\leq
		\sum_{v \in \calV} 
		\abs{a_v} \cdot \abs{v^\alpha} e^{\sprod{x}{v}} 
		\leq
		C^{\abs{\alpha}} 
		\sum_{v \in \calV} \abs{a_v} e^{\sprod{x}{v}},
	\end{equation}
	by \eqref{eq:26} and the Leibniz's rule we obtain \eqref{eq:12}.
	To show \eqref{eq:26}, we use Fa\`a di Bruno's formula with $F(t) = 1/t$. By Lemma \ref{lem:1} together with
	estimates \eqref{eq:11} and \eqref{eq:28} we get
	\begin{align*}
		\bigg|
		\partial^{\sigma} 
		\bigg\{ \frac{1}{\sum_{v \in \calV} b_v e^{\sprod{z}{v}}} \bigg\}
		\bigg|
		&\leq
		\sum_{\pi \in \Pi_\sigma} c_\pi m! \Big(\sum_{v \in \calV} b_v e^{\sprod{x}{v}} \Big)^{-m-1}
		\prod_{j=1}^m \Big| \partial^{B_j}\Big\{\sum_{v \in \calV} b_v e^{\sprod{z}{v}} \Big\}\Big| \\
		&\leq
		C^{\abs{\sigma}}
		\frac{1}{\sum_{v \in \calV} b_v e^{\sprod{x}{v}}}
		\sum_{\pi \in \Pi_\sigma} c_\pi m! \\
		&\leq
		C^{\abs{\sigma} + 1}
		\frac{1}{\sum_{v \in \calV} b_v e^{\sprod{x}{v}}},
	\end{align*}
	where in the last inequality we have used \eqref{eq:22}.
\end{proof}

\section{Heat kernels}
\label{sec:2}
In this section we show the asymptotic behaviour of the $n$'th step transition density of an irreducible random walk
on the integer lattice $\ZZ^d$. Before we state and proof the main theorem, let us present the following example.
\begin{example}
	Let $p$ be the transition function of the simple random walk on $\ZZ$. If $x \equiv n \pmod 2$ then
	\[
		p(n; x) = \frac{1}{2^n} \frac{n!}{(\frac{n-x}{2})! (\frac{n+x}{2})!}. 
	\]
	Let us recall Stirling's formula
	\[
		n! = \sqrt{2 \pi} n^{n+\frac{1}{2}} e^{-n} \big(1 + \calO(n^{-1})\big).
	\]
	Hence, we have
	\begin{align*}
		p(n; x) 
		& = \frac{1}{\sqrt{2\pi}} \frac{n^{n+\frac{1}{2}}}{(n-x)^{\frac{n-x+1}{2}} (n+k)^{\frac{n+x+1}{2}}}
		\big(1 + \calO((n-x)^{-1}) + \calO((n+x)^{-1})\big) \\
		&= \frac{1}{\sqrt{2\pi n}} (1-\delta^2)^{-\frac{1}{2}} e^{-n\phi(\delta)}
		\big(1+\calO(n^{-1} \dist(\delta, \{-1, 1\})^{-1})\big)
	\end{align*}
	where $\delta = \frac{x}{n}$ and 
	$\phi(\delta) = \frac{1}{2} (1-\delta) \log (1-\delta) + \frac{1}{2} (1+\delta) \log (1+\delta)$.
\end{example}
\begin{theorem}
	\label{th:7}
	Let $p$ be an irreducible random walk on $\ZZ^d$. Let $r$ be its period and $X_0, \ldots, X_{r-1}$ the partition
	of $\ZZ^d$ into aperiodic classes. There is $\eta \geq 1$ such that for each $j \in \{0, 1, \ldots, r-1\}$,
	$n \in \NN$ and $x \in X_j$, if $n \equiv j \pmod r$ then
	\[
		p(n; x)
		=
		(2 \pi n)^{-\frac{d}{2}} (\det B_s)^{-\frac{1}{2}} e^{-n\phi(\delta)} 
		\big(
			r + \calO\big(n^{-1} \dist(\delta, \partial \mathcal{M})^{-2\eta}\big)
		\big),
	\]
	otherwise $p(n; x) = 0$, where $\delta = \frac{x}{n}$, $s = s(\delta)$ satisfies $\nabla \log \kappa(s) = \delta$,
	and
	\[
		\phi(\delta) = \max\big\{\sprod{u}{\delta} - \log \kappa(u) : u \in \RR^d \big\}.
	\]
\end{theorem}
\begin{proof}
	Using the Fourier inversion formula we can write
	\begin{equation}
		\label{eq:37}
		p(n; x)=\bigg(\frac{1}{2\pi}\bigg)^d \int_{\mathscr{D}_d} 
		\kappa(i\theta)^n e^{-i\sprod{\theta}{x}} \dth
	\end{equation}
	where $\mathscr{D}_d = [-\pi, \pi)^d$. If $\theta_0 \in \calU$ then $\kappa(i \theta_0) = e^{it}$ for
	some $t \in [-\pi, \pi)$ where $\calU$ is defined in \eqref{eq:43}. Since $\kappa(i\theta_0)$ is a convex combination
	of complex numbers from the unit circle, $\kappa(i \theta_0) = e^{i t}$ if and only if
	$e^{i \sprod{\theta_0}{v}} = e^{it}$ for each $v \in \calV$. In particular,
	\begin{align*}
		e^{i n t} p(n; x) 
		& =
		\bigg(\frac{1}{2\pi}\bigg)^d \int_{\mathscr{D}_d}
        \kappa(i\theta + i \theta_0)^n e^{-i\sprod{\theta}{x}} \dth \\
		& =
		e^{i \sprod{\theta_0}{x}}
		p(n; x),
	\end{align*}
	thus, whenever $p(n; x) > 0$, we have
	\begin{equation}
		\label{eq:25}
		e^{i n t} = e^{i \sprod{\theta_0}{x}}.
	\end{equation}
	Hence, by \eqref{eq:6},
	\[
		e^{i n t } = e^{i \sprod{\theta_0}{x}} = e^{i (n + r) t},
	\]
	which implies that $e^{i t}$ is $r$'th root of unity. In particular, the set $\calU$ has the cardinality $r$.
	Next, we claim that
	\begin{claim}
		\label{clm:1}
		For any $u \in \RR^d$,
		\[
			\int_{\mathscr{D}_d} \kappa(i \theta)^n e^{-i \sprod{\theta}{x}} \dth
			=
			\int_{\mathscr{D}_d} \kappa(u + i\theta)^n e^{-\sprod{u + i \theta}{x}} \dth.
		\]
	\end{claim}
	To see this, we observe that for $y \in \ZZ^d$ we have
	\begin{align*}
		\int_{\mathscr{D}_d} e^{i \sprod{\theta}{y}} e^{-i\sprod{\theta}{x}} \dth
		& =
		\prod_{j = 1}^d
		\int_{-\pi}^\pi e^{i \theta_j y_j} e^{-i \theta_j x_j} {\: \rm d}\theta_j \\
		& =
		\prod_{j = 1}^d
		\int_{-\pi}^\pi e^{(u_j+i \theta_j) y_j} e^{-(u_j + i \theta_j) x_j} {\: \rm d}\theta_j \\
		& =
		\int_{\mathscr{D}_d} e^{\sprod{u + i\theta}{y}} e^{-\sprod{u + i\theta}{x}} \dth.
	\end{align*}
	Therefore,
	\begin{align*}
		\int_{\mathscr{D}_d} \kappa(i\theta)^n e^{-i\sprod{\theta}{x}} \dth
		& =
		\sum_{v_1, \ldots, v_n \in \calV} 
		\prod_{j=1}^n p(v_j)
		\int_{\mathscr{D}_d} e^{i\sprod{\theta}{\sum_{j=1}^n v_j}} e^{-i\sprod{\theta}{x}} \dth \\
		& =
		\sum_{v_1, \ldots, v_n \in \calV}
		\prod_{j=1}^n p(v_j)
		\int_{\mathscr{D}_d} e^{\sprod{u + i\theta}{\sum_{j=1}^n v_j}} e^{-\sprod{u+i\theta}{x}} \dth \\
		& =
		\int_{\mathscr{D}_d} \kappa(u + i \theta)^n e^{-\sprod{u+i\theta}{x}} \dth.
	\end{align*}

	We notice that if $p(n; x) > 0$ then $\delta = \frac{x}{n} \in \overline{\calM}$.
	Since $\dist(\delta, \partial \calM) > 0$, by Theorem \ref{th:1}, there is the unique $s=s(\delta)$ such that
	$\nabla \log \kappa(s) = \delta$. Hence, by Claim \ref{clm:1}, we can write
	\[
		p(n; x) = \bigg(\frac{1}{2 \pi}\bigg)^d e^{-n \phi(\delta)}
		\int_{\mathscr{D}_d} 
		\bigg(\frac{\kappa(s + i \theta)}{\kappa(s)}\bigg)^n e^{-i \sprod{\theta}{x}} \dth.
	\]
	Let $\epsilon > 0$ be small enough to satisfy \eqref{eq:2}, \eqref{eq:21} and \eqref{eq:23}. We set
	\[
		\mathscr{D}_d^\epsilon = \bigcap_{\theta_0 \in \calU} 
		\big\{\theta \in [-\pi, \pi)^d : \norm{\theta - \theta_0}_2 \geq \epsilon \big\}.
	\]
	Then the integral over $\mathscr{D}_d^\epsilon$ is negligible. To see this, we write
	\begin{align}
		\nonumber
		1 - \bigg\lvert \frac{\kappa(s+i\theta)}{\kappa(s)} \bigg\rvert^2 
		& = 
		1 - \sum_{v, v' \in \calV} \frac{p(v) e^{\sprod{s+i\theta}{v}}}{\kappa(s)} \cdot 
		\frac{p(v')e^{\sprod{s-i\theta}{v'}}}{\kappa(s)} \\
		\label{eq:5}
		& =
		2 \sum_{v, v' \in \calV} 
		\frac{p(v) e^{\sprod{s}{v}}}{\kappa(s)} \cdot \frac{p(v') e^{\sprod{s}{v'}}}{\kappa(s)} 
		\Big(\sin \Big\langle \frac{\theta}{2}, v-v'\Big\rangle\Big)^2.
	\end{align}
	Now, we need the following estimate.
	\begin{claim}
		\label{clm:2}
		For every $v_0 \in \calV$, there is $\xi > 0$ such that for all $\theta \in \mathscr{D}_d^\epsilon$
		there is $v \in \calV$ satisfying
		\begin{equation}
			\label{eq:15}
			\Big| \sin \Big\langle \frac{\theta}{2}, v - v_0 \Big\rangle \Big|
			\geq
			\xi.
		\end{equation}
	\end{claim}
	For the proof, we assume to contrary that for some $v_0 \in \calV$ and all $m \in \NN$ there is
	$\theta_m \in \mathscr{D}_d^\epsilon$ such that for all $v \in \calV$
	\[
		\Big|
		\sin \Big\langle \frac{\theta_m}{2}, v - v_0 \Big\rangle
		\Big|
		\leq \frac{1}{m}.
	\]
	By compactness of $\mathscr{D}_d^\epsilon$, there is a subsequence $(\theta_{m_k} : k \in \NN)$ convergent to
	$\theta' \in \mathscr{D}_d^\epsilon$. Then for all $v \in \calV$
	\[
		\sin \Big\langle \frac{\theta'}{2}, v-v_0 \Big\rangle = 0,
	\]
	and hence
 	\[
		\kappa(i\theta') = e^{i \sprod{\theta'}{v_0}}
	\]
	which is impossible since $\theta' \in \mathscr{D}_d^\epsilon$.

	In order to apply Claim \ref{clm:2}, we select any $v_0$ satisfying
	\[
		\sprod{v_0}{s} = \max\big\{\sprod{v}{s} : v \in \calV\big\},
	\]
	thus $e^{\sprod{s}{v_0}} \geq \kappa(s)$. By Claim \ref{clm:2} and \eqref{eq:5}, for each
	$\theta \in \mathscr{D}_d^\epsilon$ there is $v \in \calV$ such that
	\[
		1 - \bigg| \frac{\kappa(s+i\theta)}{\kappa(s)} \bigg|^2
		\geq
		2 p(v_0) \frac{p(v) e^{\sprod{s}{v}}}{\kappa(s)} \xi^2.
	\]
	Although $v$ may depend on $\theta$, by Theorem \ref{th:2}, there are $C > 0$ and $\eta \geq 1$
	such that for all $\theta \in \mathscr{D}_d^\epsilon$
	\[
		1 - \bigg| \frac{\kappa(s + i\theta)}{\kappa(s)} \bigg|^2
		\geq
		C \dist(\delta, \partial \calM)^\eta.
	\]
	Hence,
	\[
		\bigg| \frac{\kappa(s + i \theta)}{\kappa(s)} \bigg|^2
		\leq
		1 - C \dist(\delta, \partial \calM)^\eta
		\leq
		e^{-C \dist(\delta, \partial \calM)^\eta}.
	\]
	Since
	\[
		n \dist(\delta, \partial \calM)^\eta 
		= n^{\frac{1}{2}} \big(n \dist(\delta, \partial \calM)^{2 \eta}\big)^{\frac{1}{2}}
	\]
	we obtain that
	\[
		e^{-C n \dist(\delta, \partial \calM)^\eta}
		\leq C' n^{-\frac{d}{2}-1} \dist(\delta, \partial \calM)^{-2\eta}, 
	\]
	provided $n$ is large enough. Finally, since $\det B_s \leq 1$ we conclude that
	\[
		\int_{\mathscr{D}_d^\epsilon} \bigg|\frac{\kappa(s+i\theta)}{\kappa(s)}\bigg|^n \dth
		\leq
		C n^{-\frac{d}{2}-1} (\det B_s)^{-\frac{1}{2}} \dist(\delta, \partial \calM)^{-2\eta}.
	\]
	Next, let us consider the integral over
	\begin{equation}
		\label{eq:4}
		\bigcup_{\theta_0 \in \calU} \big\{\theta \in [-\pi, \pi)^d : \norm{\theta - \theta_0}_2 < \epsilon \big\}.
	\end{equation}
	By taking $\epsilon$ satisfying
	\begin{equation}
		\label{eq:2}
		\epsilon < \min\bigg\{\frac{\norm{\theta_0 - \theta_0'}_2}{2} : \theta_0, \theta_0' \in \calU\bigg\},
	\end{equation}
	we guarantee that the sets in \eqref{eq:4} are disjoint. Moreover, for any $\theta_0 \in \calU$, by the change of
	variables and \eqref{eq:25} we get 
	\begin{align*}
		\int_{\norm{\theta - \theta_0}_2 < \epsilon}
		\bigg(\frac{\kappa(s + i\theta)}{\kappa(s)}\bigg) ^n
		e^{-i\sprod{\theta}{x}} \dth
		& =
		\int_{\norm{\theta}_2 < \epsilon} 
		\bigg( \frac{\kappa(s + i\theta + i\theta_0)}{\kappa(s)} \bigg)^n
		e^{-i\sprod{\theta+\theta_0}{x}}
		\dth\\
		& =
		\int_{\norm{\theta}_2 < \epsilon}
		\bigg(\frac{\kappa(s + i \theta)}{\kappa(s)}\bigg)^n
		e^{-i\sprod{\theta}{x}} \dth.
	\end{align*}
	Therefore,
	\[
		\sum_{\theta_0 \in \calU}
		\int_{\norm{\theta-\theta_0}_2 < \epsilon}
		\bigg(\frac{\kappa(s + i \theta)}{\kappa(s)}\bigg)^n
		e^{-i\sprod{\theta}{x}} \dth
		=
		r
		\int_{\norm{\theta}_2 < \epsilon}
		\bigg(
		\frac{\kappa(s + i \theta)}{\kappa(s)}
		\bigg)^n
		e^{-i\sprod{\theta}{x}}
		\dth.
	\]
	Further, by \eqref{eq:11}, a function $z \mapsto \Log \kappa(z)$, where $\Log$ denotes the principal value of the 
	complex logarithm, is an analytic function in a strip $\RR^d + i B$ where 
	\[
		B = \Big\{b \in \RR^d : \norm{b}_2 < \big(2 \cdot \max\{\norm{v}_2 : v \in \calV\}\big)^{-1} \Big\}.
	\]
	Since for any $u \in \RR^d$ we have
	\[
		D_u^2 \Log \kappa(z) = 
		\frac{1}{2} 
		\sum_{v, v' \in \calV}
		\frac{p(v) e^{\sprod{z}{v}}}{\kappa(z)}
		\cdot
		\frac{p(v') e^{\sprod{z}{v'}}}{\kappa(z)}
		\sprod{u}{v-v'}^2,
	\]
	by Lemma \ref{lem:2}, there is $C > 0$ such that for all $\sigma \in \NN^d$ and $a + i b \in \RR^d + i B$
	\begin{equation}
		\label{eq:14}
		\big|\der{\sigma} \big( D_u^2 \Log \kappa \big)(a+ib) \big|
		\leq
		C^{\abs{\sigma} + 1} \sigma! B_a(u, u).
	\end{equation}
	If 
	\begin{equation}
		\label{eq:21}
		\epsilon < \big(2 \cdot \max\{\norm{v}_2 : v \in \calV\}\big)^{-1},
	\end{equation}
	then for $\norm{\theta}_2 < \epsilon$ we can define 
	\[
		\psi(s, \theta) = \Log \kappa(s+i\theta) - \log \kappa(s) - i\sprod{\theta}{\delta} 
		+ \frac{1}{2}B_s(\theta, \theta).
	\]
	Hence,
	\[
		\int_{\norm{\theta}_2 < \epsilon} \bigg(\frac{\kappa(s+i\theta)}{\kappa(s)}\bigg)^n e^{-i \sprod{\theta}{x}} \dth
		=
		\int_{\norm{\theta}_2 < \epsilon} e^{-\frac{n}{2} B_s(\theta, \theta)} e^{n\psi(s, \theta)} \dth,
	\]
	and to finish the proof of theorem it is enough to show
	\begin{claim}
		\label{clm:4}
		\begin{multline}
			\label{eq:17}
			\int_{\norm{\theta}_2 < \epsilon}
			e^{-\frac{n}{2}B_s(\theta, \theta)}
			e^{n \psi(s, \theta)} \dth \\
			=
			(2\pi n)^{\frac{d}{2}}
			(\det B_s)^{-\frac{1}{2}} 
			\Big(1+ \calO(n^{-1} \dist(\delta, \partial \calM)^{-2 \eta})\Big).
		\end{multline}
	\end{claim}
	Using the integral form for the reminder, we get
	\[
		\psi(s, \theta) = -\frac{i}{2} \int_0^1 (1 - t)^2 D_\theta^3 \Log \kappa(s + i t \theta) {\: \rm d}t.
	\]
	In view of \eqref{eq:14}, there is $C > 0$ such that for all $s \in \RR^d$ and $\theta \in B$
	\begin{equation}
		\label{eq:27}
		\abs{\psi(s, \theta)} \leq C \norm{\theta}_2 B_s(\theta, \theta).
	\end{equation}
	Therefore, by choosing
	\begin{equation}
		\label{eq:23}
		\epsilon < \bigg(4 \cdot \sup\bigg\{\frac{\abs{\psi(a, b)}}{\norm{b}_2 B_a(b, b)} : 
		a \in \RR^d, b \in B\bigg\}\bigg)^{-1},
	\end{equation}
	if $\norm{\theta}_2 < \epsilon$ then we may estimate
	\begin{equation}
		\label{eq:18}
		\abs{ \psi(s, \theta) }
		\leq
		\frac{1}{4}
		B_s(\theta, \theta).
	\end{equation}
	Next, we write
	\begin{align*}
		e^{n\psi(s, \theta)}
		& = \left(e^{n\psi(s, \theta)} - 1 - n\psi(s, \theta)\right)
		+ \left(n\psi(s, \theta) - n \frac{D_\theta^3 \psi(s, 0)}{3!}\right) \\
		&\quad + n \frac{D_\theta^3 \psi(s, 0)}{3!} + 1,
	\end{align*}
	and split \eqref{eq:17} into four corresponding integrals.

	Since for $a \in \CC$
	\[
		\abs{e^a - 1 - a} \leq \frac{\abs{a}^2}{2} e^{\abs{a}},
	\]
	by \eqref{eq:27} and \eqref{eq:18}, the first integrand can be estimated as follows
	\begin{align*}
		\left| e^{n\psi(s, \theta)} - 1 - n\psi(s, \theta) \right|
		& \leq \frac{1}{2} e^{\frac{n}{4} B_s(\theta, \theta)} \big(n \psi(s, \theta)\big)^2 \\
		& \leq C e^{\frac{n}{4} B_s(\theta, \theta)} n^2 \norm{\theta}_2^2 B_s(\theta, \theta)^2.
	\end{align*}
	Because
	\begin{equation}
		\label{eq:32}
		\norm{\theta}_2^2
		\leq \vnorm{B_s^{-1}} B_s(\theta, \theta),
	\end{equation}
	we obtain
	\begin{align}
		\nonumber
		\bigg\lvert
		& \int_{\norm{\theta}_2 < \epsilon} 
		e^{-\frac{n}{2} B_{s}(\theta, \theta)} 
		\big( e^{n\psi(s, \theta)} - 1 - n\psi(s, \theta) \big) \dth 
		\bigg\rvert \\
		\nonumber
		& \hspace{12em} \leq
		C n^2 \vnorm{B_s^{-1}}
		\int_{\norm{\theta}_2 < \epsilon}
		e^{-\frac{n}{4} B_s(\theta, \theta)}
		B_s(\theta, \theta)^3
		\dth \\
		\label{eq:29}
		& \hspace{12em} \leq
		C n^{-\frac{d}{2}-1} (\det B_{s})^{-\frac{1}{2}} \big\lVert B_{s}^{-1} \big\rVert.
	\end{align}
	Furthermore, by \eqref{eq:14},
	\begin{align*}
		\bigg| \psi(s, \theta) - \frac{D_\theta^3 \psi(s, 0)}{3!} \bigg|
		& =
		\left|
		\frac{1}{3!} \int_0^1 (1-t)^3 D_\theta^4 \Log \kappa(s + i t\theta) {\: \rm d}t
		\right| \\
		& \leq C \norm{\theta}_2^2 B_s(\theta, \theta),
	\end{align*}
	which together with \eqref{eq:32} implies
	\begin{align}
		\nonumber
		& \bigg| \int_{\norm{\theta}_2 < \epsilon} e^{-\frac{n}{2} B_{s}(\theta, \theta)}
		n \bigg( \psi(s, \theta) - \frac{D_\theta^3 \psi(s, 0)}{3!} \bigg) \dth	\bigg| \\
		\nonumber
		& \hspace{12em} \leq 
		C n \vnorm{B_s^{-1}}
		\int_{\norm{\theta}_2 < \epsilon} 
		e^{-\frac{n}{2} B_s(\theta, \theta)} 
		B_s(\theta, \theta)^2 \dth \\
		\label{eq:31}
		& \hspace{12em} \leq
		C n^{-\frac{d}{2} - 1} (\det B_{s})^{-\frac{1}{2}} \big\lVert B_{s}^{-1} \big\rVert.
	\end{align}
	The third integral is equal zero. The last one, by \eqref{eq:32}, we can estimate
	\begin{align}
		\nonumber
		& \bigg\lvert
		\int_{\norm{\theta}_2 < \epsilon} 
		e^{-\frac{n}{2} B_s(\theta, \theta)} \dth
		- \int_{\RR^d}
		e^{-\frac{n}{2} B_s(\theta, \theta)} \dth
		\bigg\rvert \\
		\nonumber
		& \hspace{12em} \leq
		e^{-\frac{n}{4} \epsilon^2 \vnorm{B_s^{-1}}^{-1}}
		\int_{\RR^d} e^{-\frac{n}{4} B_s(\theta, \theta)} \dth \\
		\label{eq:33}
		& \hspace{12em} \leq
		C n^{-\frac{d}{2} -1} (\det B_s)^{-\frac{1}{2}} \vnorm{B_s^{-1}}.
	\end{align}
	By putting estimates \eqref{eq:29}, \eqref{eq:31} and \eqref{eq:33} together, we obtain
	\begin{align*}
		\int_{\norm{\theta}_2 < \epsilon} e^{-\frac{n}{2}B_s(\theta, \theta)} e^{n\psi(s, \theta)} \dth
		=
		n^{-\frac{d}{2}}
		(\det B_s)^{-\frac{1}{2}} 
		\big(
		(2\pi)^{\frac{d}{2}}
		+
		\calO\big(n^{-1} \vnorm{B_s^{-1}}\big)\big).
	\end{align*}
	Finally, by \eqref{eq:1} and Theorem \ref{th:2}, there is $C > 0$ such that for all $\delta \in \calM$
	and any $u \in \RR^d$
	\begin{equation}
		\label{eq:40}
		B_0(u, u) \geq B_s(u, u) \geq C \dist(\delta, \partial \mathcal{M})^{2 \eta} B_0(u, u).
	\end{equation}
	Hence,
	\[
		\vnorm{B_s^{-1}} = \Big(\min\{B_s(u, u) : \norm{u}_2 = 1\}\Big)^{-1}
		\leq C \dist(\delta, \partial \calM)^{-2\eta},
	\]
	which concludes the proof of Claim \ref{clm:4}.
\end{proof}

Although, the asymptotic in Theorem \ref{th:7} is uniform on a large region with respect to $n$ and $x$,
it depends on the implicit function $s(\delta)$. By \eqref{eq:40}, we may estimate
\[
	1 \geq \det B_s \geq C \dist(\delta, \partial \calM)^{2 r \eta}.
\]
Since $\calM \ni \delta \mapsto s(\delta)$ is
real analytic, for each $\epsilon > 0$ there is $C_\epsilon > 0$ such that if
$\dist(\delta, \partial \calM) \geq \epsilon$ then
\begin{equation}
	\label{eq:34}
	\norm{s}_1 \leq C_\epsilon \norm{\delta - \delta_0}_1,
\end{equation}
and
\[
	\Big|
	( \det B_s \big)^{-\frac{1}{2}} 
	-
	( \det B_0 \big)^{-\frac{1}{2}}
	\Big|
	\leq
	C_\epsilon \norm{\delta - \delta_0}_1
\]
where $\delta_0 = \sum_{v \in \calV} p(v) v$. In the most applications the following form of the asymptotic of $p(n; x)$
is sufficient.
\begin{corollary}
	\label{cor:1}
	For every $\epsilon > 0$, $j \in \{0, \ldots, r-1\}$, $n \in \NN$ and $x \in X_j$, 
	if $n \equiv j \pmod r$ then
	\[
		p(n; x) =
		(2\pi n)^{-\frac{d}{2}} (\det B_0)^{-\frac{1}{2}}
		e^{-n \phi(\delta)} 
		\left(r + \calO(\norm{\delta - \delta_0}_1) + \calO(n^{-1})\right),
	\]
	otherwise $p(n; x) = 0$, provided that $\dist(\delta, \partial \calM) \geq \epsilon$.
\end{corollary}
\begin{remark}
	\label{rem:1}
	It is not possible to replace $\phi(\delta)$ by $\frac{1}{2} B_0^{-1}(\delta - \delta_0, \delta - \delta_0)$
	without introducing an error term of a very different nature. Indeed, by \eqref{eq:10},
	\[
		e^{-n \phi(\delta)} = e^{-\frac{n}{2}B_0^{-1}(\delta-\delta_0, \delta - \delta_0)} 
		e^{\calO(n \norm{\delta - \delta_0}^3)}.
	\]
	Since $\delta_0 \in \calM$, if $\delta$ approaches $\partial \calM$ then $n \norm{\delta - \delta_0}^3$ cannot be
	small. Notice that the third power may be replaced by an higher degree if the random walk has vanishing moments.
	In particular, for the simple random walk on $\ZZ^d$ (see Example \ref{ex:1}), for all $\epsilon > 0$,
	$x \in \ZZ^d$ and $n \in \NN$, if $\norm{x}_1 + n \in 2 \NN$ then
	\[
		p(n; x) =
		(2\pi)^{-\frac{d}{2}}
		\left(\frac{d}{n}\right)^{\frac{d}{2}}
		e^{-\frac{d}{2n} \norm{x}^2_2}
		\big(2 + \calO(n \abs{\delta}^4) + \calO(n^{-1})\big)
	\]
	otherwise $p(n; x) = 0$, uniformly with respect to $n$ and $x$ provided that $\norm{x}_1 \leq (1-\epsilon) n$.
\end{remark}

\begin{remark}
	It is relatively easy to obtain a global upper bound: for all $n \in \mathbb{N}$ and $x \in \mathbb{Z}^d$
	\[
		p(n; x) \leq e^{-n\phi(\delta)}.
	\]
	Indeed, by Claim \ref{clm:1}, for $u \in \RR^d$, we have
	$$
	p(n; x) = \bigg(\frac{1}{2\pi}\bigg)^d \kappa(u)^n e^{-\sprod{u}{x}}
	\int_{\mathscr{D}_d} \bigg( \frac{\kappa(u+i\theta)}{\kappa(u)} \bigg)^n
	e^{-i \sprod{\theta}{x}} \dth.
	$$
	Hence, by Theorem \ref{th:1},
	\[
		p(n; x) 
		\leq \min\big\{ \kappa(u)^n e^{-\sprod{u}{x}} : u \in \mathbb{R}^d \big\} 
		\leq e^{-n\phi(\delta)}.
	\]
\end{remark}

\appendix
\section{Applications}
\label{sec:1}
In this section we apply Corollary \ref{cor:1} to simple random walks on triangular and hexagonal lattices.

\subsection{The triangular lattice}
The triangular lattice consists of the set of points 
\[
	L = \ZZ \lambda_1 \oplus \ZZ \lambda_2,
\]
where $\lambda_1 = (-1/2, \sqrt{3}/2)$, $\lambda_2 = (1/2, \sqrt{3}/2)$. Let $\tau: L \rightarrow \{0, 1, 2\}$
be defined by setting
\[
	\tau\big(j \lambda_1 + j' \lambda_2\big) 
	\equiv j + 2j' \pmod 3.
\]
Each point $x \in L$ has six closest neighbours, namely,
\begin{alignat*}{3}
	x + \lambda_1, \quad &x - \lambda_1, \quad &x + \lambda_1 - \lambda_2, \\
	x + \lambda_2, \quad &x - \lambda_2, \quad &x - \lambda_1 + \lambda_2.
\end{alignat*}
Let $p$ be the transition function of the simple random walk on $L$. Observe that the mapping 
\[
	L \ni j \lambda_1 + j' \lambda_2 \mapsto j e_1 + j' e_2 \in \ZZ^2
\]
allows us instead of $p$ to work with a transition function $\tilde{p}$ such that
\[
	\tilde{p}(e_1) = \tilde{p}(-e_1) = \tilde{p}(e_1 - e_2) 
	= \tilde{p}(e_2) = \tilde{p}(-e_2) = \tilde{p}(-e_1+e_2) = \frac{1}{6}.
\]
Then the corresponding set $\calM \subset \RR^2$ is the interior of
\[
	\conv \big\{e_1, e_2 , e_1-e_2, -e_1, -e_2, -e_1+e_2\big\}.
\]
Moreover, for $u \in \RR^2$
\begin{align*}
	\kappa(u) & = 
			\frac{1}{6} e^{u_1} + \frac{1}{6} e^{-u_1} + \frac{1}{6} e^{u_2} + \frac{1}{6} e^{-u_2}
			+\frac{1}{6} e^{u_1 - u_2} + \frac{1}{6} e^{-u_1+u_2} \\
			& =
			\frac{1}{3} \cosh u_1 + \frac{1}{3} \cosh u_2 + \frac{1}{3} \cosh (u_1 - u_2).
\end{align*}
In particular, $\calU = \{0\}$. Next, by using \eqref{eq:1} we can calculate the quadratic form $B_0$,
\[
	B_0 =
	\frac{1}{3} 
	\begin{pmatrix}
		2 & -1 \\
		-1 & 2 \\
	\end{pmatrix},
\]
thus $\det B_0 = 1/3$. Finally, by applying Corollary \ref{cor:1} to $\tilde{p}$ we obtain the following precise
asymptotic of $p$: for every $\epsilon > 0$ and all $n \in \NN$ and $x \in L$
\begin{equation}
	\label{eq:20}
	p(n; x)
	=
	(2\pi n)^{-1} e^{-n \phi(\delta)} \big(\sqrt{3} + \calO(\norm{\delta}_1) + \calO(n^{-1})\big),
\end{equation}
uniformly with respect to $n$ and $x$ such that $\dist(\delta, \partial \calM) \geq \epsilon$
where for $x =j \lambda_1 + j' \lambda_2$ we have set 
\[
	\delta = \frac{j}{n} e_1 + \frac{j'}{n} e_2.
\]
The reader may wish to compare the asymptotic \eqref{eq:20} with know results (see \cite[Example 2]{ks2},
\cite{ks1} and \cite[Section 7.2]{ikk}).

\subsection{The hexagonal lattice}
The hexagonal lattice $H$ one may obtain from the triangular lattice by removing all vertices $x \in L$ such that 
$\tau(x) = 1$. Each vertex $x \in H$ has three neighbours,
\[
	\begin{cases}
	x + \lambda_2, \quad x - \lambda_1, \quad x - \lambda_2 + \lambda_1,& \quad\text{if } \tau(x) = 0 \\
	x + \lambda_1, \quad x - \lambda_2, \quad x - \lambda_1 + \lambda_2,& \quad\text{if } \tau(x) = 2.
	\end{cases}
\]
Let $p$ be the transition function of the simple random walk on $H$, i.e.
$p(x, y) = 1/3$ if $x$ and $y$ are closest neighbours. Observe that $p$ is irreducible and periodic with period $r = 2$.
We have
\[
	X_0 = \{x \in H : \tau(x) = 0\}, \quad X_1 = \{x \in H : \tau(x) = 2\}.
\]
Let us consider a new random walk given by a transition function $q$
\[
	q(x, y) = \sum_{\atop{u \sim x}{u \sim y}} p(x, u) p(u, y),
\]
where the sum is taken over $u \in H$ being a common neighbour of $x$ and $y$. It is easy to check that
$q(x, x) = 1/3$ and $q(x, y) = 1/9$ where $y$ belongs to the set
\begin{alignat*}{3}
	&x + \lambda_1 + \lambda_2, \quad &x + 2\lambda_1 - \lambda_2, \quad &x+\lambda_1 - 2\lambda_2,\\
	&x - \lambda_1 - \lambda_2, \quad &x - \lambda_1 + 2\lambda_2, \quad &x-2 \lambda_1 + \lambda_2.
\end{alignat*}
If $x \in X_0$ then
\[
	p(n; 0, x) = 
	\begin{cases}
		q(n/2; 0, x) & \text{if } n \equiv 0 \pmod 2,\\
		0 & \text{otherwise.}
	\end{cases}
\]
If $x \in X_1$ and $n \equiv 1 \pmod 2$ then
\begin{align*}
	p(n; 0, x) &= \frac{1}{3} q((n-1)/2; 0, x+\lambda_1) \\
	&\phantom{=}+ \frac{1}{3} q((n-1)/2; 0, x-\lambda_2) \\
	&\phantom{=}+ \frac{1}{3} q((n-1)/2; 0, x+\lambda_1-\lambda_2)
\end{align*}
otherwise $p(n; 0, x) = 0$. First, we find the asymptotic of $q$. We notice that $q$ is the transition function
of an irreducible random walk on the triangular lattice 
\[
	L = \ZZ (\lambda_1 + \lambda_2) \oplus \ZZ (2 \lambda_2 - \lambda_1).
\]
Therefore, under the mapping
\[
	L \ni j \lambda_1 + j' \lambda_2 \mapsto \frac{2 j+ j'}{3}e_1 + \frac{-j+j'}{3} e_2 \in \ZZ^2
\]
the transition function $q$ is mapped onto $\tilde{q}$ a transition function of a random walk on the integer lattice
$\ZZ^2$ where
\begin{align*}
	& \tilde{q}(0) = \frac{1}{3},\\
	& \tilde{q}(e_1) = \tilde{q}(-e_1) = 
	\tilde{q}(e_2) = \tilde{q}(-e_2) = 
	\tilde{q}(e_1-e_2) = \tilde{q}(e_2-e_1) = \frac{1}{9}.
\end{align*}
In this case, for $u \in \RR^2$ we have
\begin{align*}
	\kappa(u) &=
	\frac{1}{3} + \frac{1}{9} e^{u_1} + \frac{1}{9} e^{-u_1} + \frac{1}{9} e^{u_2} + \frac{1}{9} e^{u_2}
	+\frac{1}{9} e^{u_1-u_2} + \frac{1}{9} e^{-u_1+u_2}\\
	&=
	 \frac{1}{3} + \frac{2}{9} \cosh u_1 + \frac{2}{9} \cosh u_2 + \frac{2}{9} \cosh (u_1-u_2).
\end{align*}
Again, the set $\calM \subset \RR^2$ is the interior of
\[
	\conv \big\{e_1, e_2 , e_1-e_2, -e_1, -e_2, -e_1+e_2\big\}.
\]
It is easy to calculate that the quadratic form $B_0$ is equal to
\[
	B_0 =
	\frac{2}{9} 
	\begin{pmatrix}
	2 & -1 \\
	-1 & 2
	\end{pmatrix},
\]
thus $\det B_0 = 4/27$. By Corollary \ref{cor:1}, for every $\epsilon > 0$, all $n \in \NN$ and
$x \in L$ 
\begin{equation}
	\label{eq:9}
	q(n; x) 
	=
	\big(2 \pi n\big)^{-1} e^{-n \phi(\delta)} \big(3\sqrt{3} + \calO(\norm{\delta}_1) + \calO(n^{-1})\big)
\end{equation}
uniformly with respect to $n$ and $x$ such that $\dist(\delta, \partial \calM) \geq \epsilon$
where for $x = j \lambda_1 + j' \lambda_2$ we have set
\begin{equation}
	\label{eq:36}
	\delta = \frac{2 j + j'}{3 n} e_1 + \frac{-j + j'}{3 n} e_2.
\end{equation}
Although, for $x \in X_1$, we need to apply \eqref{eq:9} for three times with different $x$, the
exponential factors are comparable. Indeed, for $x \in H$, $x = j \lambda_1 + j' \lambda_2$ let $\delta$ be defined
by the formula \eqref{eq:36}. Fix $\epsilon > 0$ and let us consider $x \in X_1$ and $n \in \NN$ such that 
\[
	\dist(\delta, \partial \calM) \geq \epsilon.
\]
Let $\tilde{\delta}$ be any element of a set
\[
	\delta + \bigg\{\frac{2}{3n} e_1 - \frac{1}{3n} e_2, -\frac{1}{3n}e_1 - \frac{1}{3n} e_2,
	\frac{1}{3n} e_1 - \frac{2}{3n} e_n\bigg\}.
\]
Observe that
\begin{equation}
	\label{eq:35}
	\lvert \delta - \tilde{\delta} \rvert_1 \leq \frac{1}{n}.
\end{equation}
Let us denote by $s \in \RR^2$ the unique solution to $\nabla \log \kappa(s) = \delta$, then by \eqref{eq:30}, we can
estimate
\[
	\phi(\delta) - \phi(\tilde{\delta})
	\leq
	\phi(\delta) - \sprod{\tilde{\delta}}{s} + \log \kappa(s) = \sprod{\delta - \tilde{\delta}}{s}.
\]
Hence, by \eqref{eq:34} and \eqref{eq:35}, we get
\[
	\phi(\delta) - \phi(\tilde{\delta})
	\leq
	C_\epsilon \frac{1}{n} \norm{\delta}_1.
\]
In particular,
\[
	e^{-n \phi(\tilde{\delta})} 
	= 
	e^{-n \phi(\delta)} e^{n(\phi(\delta) - \phi(\tilde{\delta}))}
	=
	e^{-n\phi(\delta)} \big(1 + \calO(\norm{\delta}_1)\big).
\]
Now, we are ready to apply \eqref{eq:9} to obtain the precise asymptotic of $p$. For every $\epsilon > 0$ and all
$x \in H$, $n \in \NN$, $j \in \{0, 1\}$, if $x \in X_j$ then
\begin{align*}
	p(n; 0, x) = 
	\begin{cases}
		(2\pi n)^{-1} e^{-n\phi(\delta)} \big(3 \sqrt{3} + \calO(\norm{\delta}_1) + \calO(n^{-1})\big),
		& \text{if } n \equiv j \pmod 2,\\
		0, & \text{otherwise,}
	\end{cases}
\end{align*}
uniformly with respect to $n$ and $x$ such that $\dist(\delta, \partial \calM) \geq \epsilon$. The reader may compare
the above asymptotic with \cite[Section 7.3]{ikk}, \cite[Example 3]{ks2} and \cite{ks1}.

\begin{bibliography}{heat}
\bibliographystyle{plain}

\newcommand{\SL}{{\operatorname{SL}}} \newcommand{\SU}{{\operatorname{SU}}}
  \newcommand{\UU}{{\operatorname{U}}} \newcommand{\Sp}{{\operatorname{Sp}}}
  \newcommand{\SO}{{\operatorname{SO}}}
\begin{thebibliography}{10}

\bibitem{bct}
A.~Bendikov, W.~Cygan, and B.~Trojan.
\newblock Asymptotic formula for subordinated random walks.
\newblock arXiv: 1504.01759, 2015.

\bibitem{dav}
E.B. Davis.
\newblock Large deviation for heat kernels on graphs.
\newblock {\em J. London Math. Soc.}, 2:65--72, 1993.

\bibitem{guiv}
Y.~Guivarc'h.
\newblock Application d'un theoreme limite local a la transience et a la
  recurrence de marches de markov.
\newblock In G.~Gabriel~Mokobodzki and D.~Pinchon, editors, {\em Th\'{e}orie du
  Potentiel, Orsay 1983}, volume 1096 of {\em Lect. Notes Math.}, pages
  301--332. Springer--Verlag, 1984.

\bibitem{ikk}
S.~Ishiwata, H.~Kawabi, and M.~Kotani.
\newblock Long time asymptotics of non-symmetric random walks on crystal
  lattice.
\newblock arxiv: 1510.05103, 2015.

\bibitem{ks2}
M.~Kotani and T.~Sunada.
\newblock Albanese maps and off diagonal long time asymptotics for the heat
  kernel.
\newblock {\em Commun. Math. Phys.}, 209:633--670, 2000.

\bibitem{ks1}
M.~Kotani and T.~Sunada.
\newblock Large deviation and the tangent cone at inifinity of a crystal
  lattice.
\newblock {\em Math. Z.}, 254:837 -- 870, 2006.

\bibitem{kSz}
A.~Kr\'amli and D.~Sz\'asz.
\newblock Random walks with internal degrees of freedom.
\newblock {\em Z. Wahrsch. Verw. Gebiete}, 63:85--95, 1983.

\bibitem{law}
G.F. Lawler.
\newblock {\em Intersections of random walks}.
\newblock Probability and Its Applications. Birkh\"auser, 1991.

\bibitem{lawlim}
G.F. Lawler and V.~Limi\v{c}.
\newblock {\em {R}andom Walks: A modern Introduction}.
\newblock Cambridge Studies in Advanced Mathematics. Cambridge University
  Press, 2010.

\bibitem{ns}
P.~Ney and F.~Spitzer.
\newblock {M}artin boundary for random walk.
\newblock {\em Trans. Amer. Math. Soc}, 121(1):116--132, 1966.

\bibitem{polya0}
G.~P\'{o}lya.
\newblock \foreignlanguage{german}{ \"{U}ber eine {A}ufgabe der
  {W}ahrscheinlichkeitsrechnung betreffend die {I}rrfahrt im {S}trassennetz}.
\newblock {\em Math. Ann.}, 84:149--160, 1921.

\bibitem{smith}
W.L. Smith.
\newblock A frequency function form of the central limit theorem.
\newblock {\em P. Camb. Phil. Soc.}, 49:462--472, 1953.

\bibitem{sp}
F.~Spitzer.
\newblock {\em Principles of random walk}.
\newblock Graduate Texts in Mathematics. Springer, 1964.

\bibitem{tr}
B.~Trojan.
\newblock Heat kernel and {G}reen function estimates for affine buildings.
\newblock arxiv.org:1310.2288, 2013.

\end{thebibliography}
\end{bibliography}
\end{document}